\documentclass[a4paper,reqno]{amsart}

\usepackage{dsfont,amsfonts,hyperref,tikz-cd}
\tikzcdset{arrow style=tikz, diagrams={>=stealth}}

\newtheorem*{maintheorem}{Main Theorem}
\newtheorem{theorem}{Theorem}[section]

\newtheorem{lemma}[theorem]{Lemma}
\newtheorem{proposition}[theorem]{Proposition}

\theoremstyle{definition}

\newtheorem{example}[theorem]{Example}

\numberwithin{equation}{section}

\DeclareMathOperator{\add}{add}

\DeclareMathOperator{\Cok}{Cok}
\DeclareMathOperator{\End}{End}
\DeclareMathOperator{\Endu}{\underline{End}}
\DeclareMathOperator{\Ext}{Ext}

\DeclareMathOperator{\fd}{fd}
\DeclareMathOperator{\fg}{fg}
\DeclareMathOperator{\fp}{fp}
\DeclareMathOperator{\fcg}{fcg}
\DeclareMathOperator{\fcp}{fcp}
\DeclareMathOperator{\fdo}{\overline{fd}}

\DeclareMathOperator{\fpu}{\underline{fp}}
\DeclareMathOperator{\Hom}{Hom}
\DeclareMathOperator{\Homo}{\overline{Hom}}
\DeclareMathOperator{\Homu}{\underline{Hom}}
\let\Im\relax\DeclareMathOperator{\Im}{Im}
\DeclareMathOperator{\inj}{inj}

\DeclareMathOperator{\Ker}{Ker}

\DeclareMathOperator{\Mod}{Mod}
\DeclareMathOperator{\Modo}{\overline{Mod}}
\DeclareMathOperator{\Modu}{\underline{Mod}}
\DeclareMathOperator{\Ob}{Ob}
\DeclareMathOperator{\proj}{proj}

\DeclareMathOperator{\rad}{rad}

\DeclareMathOperator{\Tr}{Tr}

\newcommand{\1}{\mathds{1}}
\newcommand{\A}{\mathcal{A}}
\newcommand{\C}{\mathcal{C}}
\newcommand{\F}{\mathbb{F}}
\newcommand{\I}{\mathcal{I}}
\newcommand{\N}{\mathbb{N}}
\renewcommand{\P}{\mathcal{P}}
\newcommand{\FI}{\mathrm{FI}}
\newcommand{\VI}{\mathrm{VI}}
\newcommand{\op}{\mathrm{op}}
\newcommand{\Endc}{\End_\mathcal{C}}
\newcommand{\Enduc}{\Endu_\mathcal{C}}
\newcommand{\Extc}{\Ext^1_\mathcal{C}}
\newcommand{\Homc}{\Hom_\mathcal{C}}
\newcommand{\Homoc}{\Homo_\mathcal{C}}
\newcommand{\Homuc}{\Homu_\mathcal{C}}

\newcommand{\To}{\longrightarrow}
\newcommand{\xto}{\xrightarrow}

\newcommand{\norm}[1]{\left\lvert#1\right\rvert}
\newcommand{\set}[1]{\left\{#1\right\}}

\title[generalized Auslander-Reiten duality]
  {The generalized Auslander-Reiten duality on a module category}


\author{Pengjie Jiao}
\address{Department of Mathematics,
  China Jiliang University,
  Hangzhou 310018, PR China}
\email{jiaopjie@cjlu.edu.cn}

\begin{document}

\begin{abstract}
  We characterize the generalized Auslander--Reiten duality on the category of finitely presented modules over some certain Hom-finite category.
  Examples include the category FI of finite sets with injections, and the one VI of finite dimensional vector spaces with linear injections over a finite field.
\end{abstract}

\maketitle

\section{Introduction}

Let $k$ be a field. The Auslander--Reiten theory is a powerful tool for the representation theory of finite dimensional algebras.
In an Ext-finite abelian category, it was shown that the Auslander--Reiten duality holds if and only if there exist enough almost split sequences;
see \cite[Theorem~1.1]{LenzingZuazua2004Auslander}.
Under some weaker hypotheses, its local version in an exact category was established;
see \cite[Theorem~3.6]{LiuNgPaquette2013Almost}.

Moreover, the \emph{generalized Auslander--Reiten duality} on a Hom-finite Krull--Schmidt exact $k$-category $\A$ was introduced in \cite{Jiao2018generalized}.
It consists of a pair of full subcategories $\A_r$ and $\A_l$, and the \emph{generalized Auslander--Reiten translation functors} $\tau$ and $\tau^-$. Here, $\tau$ and $\tau^-$ are mutually quasi-inverse equivalences between stable categories of $\A_r$ and $\A_l$.

Recall that $\FI$ is the category whose objects are finite sets and morphisms are injections, and $\VI$ is the one whose objects are finite dimensional vector spaces and morphisms are linear injections over a finite field $\F_q$.
$\FI$-modules were introduced in \cite{ChurchEllenbergFarb2015FI} to study sequences of representations of symmetric groups.
We mention that finitely generated modules over $\FI$ and $\VI$ satisfy Noetherian property;
see such as \cite[Theorem~3.7]{GanLi2015Noetherian}.

We attempt to apply the Auslander--Reiten theory to the study of finitely presented $\FI$-modules and $\VI$-modules.

To meet the requirements, we consider a Hom-finite small $k$-category $\C$. We assume the class of objects in $\C$ is precisely $\N$ with $\C(j,i) = 0$ for any $i < j$, and each finitely generated $\C$-module is Noetherian.
In this case, the category $\fp \C$ of finitely presented $\C$-modules is abelian.

We characterize the generalized Auslander--Reiten duality on $\fp \C$.

\begin{maintheorem}[see Theorem~\ref{thm:gAR}]
  Let $\C$ be as above. Then
  \(
    ( \fp \C )_r = \fp \C
  \)
  and
  \(
    ( \fp \C )_l = \add
    ( \fd \C \cup \set{\mbox{injective objects in $\fp \C$}} )
  \),
  and $D \Tr$ and $\Tr D$ induce the generalized Auslander--Reiten translation functors.
\end{maintheorem}

Here, $\fd \C$ is the category of finite dimensional $\C$-modules, and $\proj \C$ is the one of finitely generated projective $\C$-modules. Moveover, $D \Tr$ and $\Tr D$ are the classical Auslander--Reiten translation.

As we wish, the result can be applied to the categories of finitely presented modules over $\FI$, $\VI$ and some certain infinite quivers; see Section~4.


The paper is organized as follows.
Section~2 includes some basics of $\C$-modules.
Section~3 is dedicated to the proof of Theorem~\ref{thm:gAR}.
In Section~4, we apply the result to $\FI$, $\VI$ and some quivers.

\section{Module category}

Let $k$ be a field. Denote by $\Mod k$ the category of $k$-modules.

Let $\C$ be a Hom-finite essentially small $k$-category. Denote by $\Ob \C$ the class of objects in $\C$, and by $\C(a,b)$ the set of morphisms in $\C$ for any $a,b \in \Ob \C$.

\subsection{Modules}

A $\C$-module $M$ over $k$ means a covariant $k$-functor $M \colon \C \to \Mod k$.
A morphism $f \colon M \to N$ of $\C$-modules means a natural transformation.
In other words, it consists of a collection of maps $f_a \colon M(a) \to N(a)$ of $k$-modules for any $a \in \Ob \C$, such that
\(
  N(\alpha) \circ f_a = f_b \circ M(\alpha)
\)
for any $\alpha \in \C(a, b)$.

Denote by $\Mod \C$ the category of $\C$-modules.
It is well known that $\Mod \C$ is an abelian $k$-category.
Given any $\C$-modules $M$ and $N$, we denote by $\Homc(M,N)$ the set of morphisms of $\C$-modules.
We have the faithful exact contravariant functor
\(
  D \colon \Mod \C \to \Mod \C^\op
\)
induced by
\(
  \Hom_k(-,k) \colon \Mod k \to \Mod k
\).

We mention the following fact;
see \cite[Section~3.7]{GabrielRoiter1992Representations}.
It implies that $\C(a,-)$ is projective and $D \C(-,a)$ is injective for any $a \in \Ob \C$.

\begin{lemma}\label{lem:proj-inj}
  For any $M \in \Mod \C$ and $a \in \C$, there exist natural isomorphisms
  \(
    \Homc (\C(a,-), M) \cong M(a)
  \)
  and
  \(
    \Homc ( M, D \C(-,a) ) \cong D M(a)
  \).
  \qed
\end{lemma}

Given a collection $\A$ of $\C$-modules, denote by $\add \A$ the full subcategory of $\Mod \C$ formed by direct summands of finite direct sums of objects in $\A$.
Set
\(
  \proj \C = \add \set{\C(a,-) \middle| a \in \Ob \C}
\)
and
\(
  \inj \C = \add \set{ D \C(-,a) \middle| a \in \Ob \C}
\).
We observe that the restriction of $D$ gives a duality
\(
  D \colon \proj \C \to \inj \C^\op
\).

A morphism $f \colon M \to N$ of $\C$-modules is called \emph{right minimal} if any endomorphism $g \in \Endc(M)$ with $f \circ g = f$ is an isomorphism. Dually, $f$ is called \emph{left minimal} if any endomorphism $h \in \Endc(N)$ with $h \circ f = f$ is an isomorphism.

Let $M$ be a $\C$-module. A right minimal epimorphism $P \to M$ with projective $P$ is called a \emph{projective cover} of $M$.
A left minimal monomorphism $M \to I$ with injective $I$ is called an \emph{injective envelope} of $M$.
It is well known that each $\C$-module admits an injective envelope;
see \cite[Theorem~3.10.10]{Popescu1973Abelian}.
Note that projective covers or injective envelopes may lie outside of $\proj \C$ or $\inj \C$.

We call $M$ \emph{finitely generated} if there exists an epimorphism $f \colon P \to M$ with $P \in \proj \C$; call $M$ \emph{finitely presented} if moreover $\Ker f$ is finitely generated.
We denote by $\fg \C$ the category of finitely generated $\C$-modules, and by $\fp \C$ the one of finitely presented $\C$-modules.


Dually, we call $M$ \emph{finitely cogenerated} if there exists a monomorphism $g \colon M \to I$ with $I \in \inj \C$; call $M$ \emph{finitely copresented} if moreover $\Cok g$ is finitely cogenerated.
We denote by $\fcg \C$ the category of finitely cogenerated $\C$-modules, and by $\fcp \C$ the one of finitely copresented $\C$-modules.


We observe that the restrictions of $D$ give dualities
\[
  D \colon \fg \C \To \fcg \C^\op
  \quad\mbox{and}\quad
  D \colon \fp \C \To \fcp \C^\op.
\]
It follows that each finitely generated $\C$-module $M$ admits a projective cover.
Indeed, since $D M$ is finitely cogenerated, we can assume $f \colon D M \to I$ is an injective envelope in $\Mod \C^\op$ with $I \in \inj \C^\op$.
Observe that both $D M(a)$ and $I(a)$ are finite dimensional for all $a \in \Ob \C$.
Then $D f \colon D I \to M$ is a projective cover.

\begin{lemma}\label{lem:Krull}
  The categories $\fg \C$ and $\fcg \C$ are Hom-finite Krull--Schmidt.
\end{lemma}

\begin{proof}
  Let $M,N \in \fg \C$.
  Then $\dim N(a)$ is finite for any $a \in \C$.
  Assume $P \to M$ is an epimorphism with $P \in \proj \C$.
  Then $\Homc(P, N)$ is finite dimensional, and so is $\Homc(M, N)$.
  Therefore, $\fg \C$ is Hom-finite.
  Moreover, it is closed under direct summands. In other words, it has split idempotents, and hence is Krull--Schmidt;
  see \cite[Corollary~4.4]{Krause2015Krull}.
  Similarly, $\fcg \C$ is also Hom-finite Krull--Schmidt.
\end{proof}

For each $\C$-module $M$, we denote by $M^*$ the $\C^\op$-module given by
\[
  \begin{array}{r@{\;}l@{\;}c@{\;}l}
    \C^\op & \To \Mod \C         & \xto{\Homc(M,-)} & \Mod k, \\
    a      & \longmapsto \C(a,-) & \longmapsto      & \Homc(M, \C(a,-)).
  \end{array}
\]
Here, the left arrow is the Yoneda embedding.
For each morphism $f \colon M \to N$ of $\C$-modules, we let
\(
  f^* \colon N^* \to M^*
\)
be the morphism of $\C^\op$-modules given by
\[
  f^*_a := \Homc (f, \C(a,-)) \colon N^* (a) \To M^* (a),
\]
for any $a \in \Ob \C$.
Then we obtain a contravariant functor
\[
  (-)^* \colon \Mod \C \To \Mod \C^\op.
\]
We mention that $(-)^*$ is left exact, since
\(
  \Homc(-, \C(a,-)) \colon \Mod \C \to \Mod k
\)
is left exact for any $a \in \Ob \C$.
We observe by Yoneda's lemma the duality
\[
  (-)^* \colon \proj \C \To \proj \C^\op.
\]

\subsection{Stable categories}

Let $\A$ be an abelian $k$-category.
Recall that a morphism $f \colon X \to Y$ in $\A$ is called \emph{projectively trivial} if for any $Z \in \Ob \A$, the induced map
\(
  \Ext_\A^1(f, Z) \colon \Ext_\A^1(Y, Z) \to \Ext_\A^1(X, Z)
\)
is the zero map;
see \cite[Section~2]{LenzingZuazua2004Auslander}.
We mention that $f$ is projectively trivial if and only if it factors through every epimorphism $f' \colon X' \to Y$.
Dually, $f$ is called \emph{injectively trivial} if for any $Z \in \Ob \A$, the induced map
\(
  \Ext_\A^1(Z, f) \colon \Ext_\A^1(Z, X) \to \Ext_\A^1(Z, Y)
\)
is the zero map.
The morphism $f$ is injectively trivial if and only if it factors through every monomorphism $f' \colon X \to Y'$.

We mention the following observation;
see \cite[Lemma~2.2]{LenzingZuazua2004Auslander} and its dual.

\begin{lemma}\label{lem:trivial}
  Let $f \colon X \to Y$ be a morphism in $\A$.
  \begin{enumerate}
    \item
      If there exists an epimorphism $g \colon P \to Y$ with projective $P$, then $f$ is projectively trivial if and only if it factors through $g$.
    \item
      If there exists a monomorphism $g \colon X \to I$ with injective $I$, then $f$ is injectively trivial if and only if it factors through $g$.
    \qed
  \end{enumerate}
\end{lemma}



Let $X, Y \in \Ob \A$. We denote by $\P(X, Y)$ the $k$-submodule of $\A(X, Y)$ formed by projectively trivial morphisms. Then $\P$ forms an ideal of $\A$. The \emph{projectively stable category} $\underline\A$ attached to $\A$ is the factor category $\A / \P$. Given a morphism $f \in \A(X, Y)$, we denote by $\underline{f}$ its image in $\underline\A$.

Dually, we denote by $\I(X, Y)$ the $k$-submodule of $\A(X, Y)$ formed by injectively trivial morphisms. The \emph{injectively stable category} $\overline\A$ attached to $\A$ is the factor category $\A / \I$. Given a morphism $f \in \A(X, Y)$, we denote by $\overline{f}$ its image in $\overline\A$.

We mention that $\Ext_\A^1(-,X)$ induces a functor
\(
  \Ext_\A^1(-,X) \colon \underline\A \to \Mod k
\),
and $\Ext_\A^1(X,-)$ induces a functor
\(
  \Ext_\A^1(X,-) \colon \overline\A \to \Mod k
\),
for any $X \in \Ob \A$.

Specially, we can consider the stable categories of $\Mod \C$.
Since $\Mod \C$ contains enough projective modules, a morphism is projectively trivial if and only if it factors through some projective module by Lemma~\ref{lem:trivial}. Similarly, a morphism is injectively trivial if and only if it factors through some injective module.

We denote by $\Modu \C$ the projectively stable category, and by $\Modo \C$ the injectively stable category.
For any $\C$-modules $M$ and $N$, we denote
\(
  \Homuc(M, N) = \Homc(M, N) / \P(M, N)
\)
and
\(
  \Homoc(M, N) = \Homc(M, N) / \I(M, N)
\).


\subsection{Auslander--Reiten formula}

Let
\(
  \delta \colon 0 \to X \to Y \to Z \to 0
\)
be an exact sequence of $\C$-modules.
The \emph{covariant defect} $\delta_*$ and the \emph{contravariant defect} $\delta^*$ are given by the following exact sequence of functors
\begin{gather*}
  0 \to \Homc(Z, -) \to \Homc(Y, -) \to \Homc(X, -) \to \delta_* \to 0,
  \\
  0 \to \Homc(-, X) \to \Homc(-, Y) \to \Homc(-, Z) \to \delta^* \to 0.
\end{gather*}
We mention that $\delta_*$ vanishes on injectively trivial morphisms, and $\delta^*$ vanishes on projectively trivial morphisms. Therefore, they induce the functors
\[
  \delta_* \colon \Modo \C \To \Mod k
  \quad \mbox{and} \quad
  \delta^* \colon \Modu \C \To \Mod k.
\]

For each finitely presented $\C$-module $M$, we fix some exact sequence
\[
  P_1(M) \overset{f_1}{\To} P_0(M) \overset{f_0}{\To} M \To 0.
\]
Here, $f_0$ and $P_0(M) \to \Im f_1$ are projective covers.
We call $\Cok f_1^*$ the \emph{transpose} of $M$, and denote by $\Tr M$;
see \cite[Section~2]{AuslanderReiten1975Representation3}.
Moreover, we have a duality
\[
  \Tr \colon \fpu \C \To \fpu \C^\op,
\]
Here, $\fpu \C$ is the full subcategory of $\Modu \C$ formed by finitely presented $\C$-modules.

We mention that if $M$ is an indecomposable non-projective finitely presented $\C$-module, then $\Tr M$ is an indecomposable non-projective $\C^\op$-module, and $\Tr \Tr M \cong M$;
see \cite[Propositon~IV.1.7]{AuslanderReitenSmalo1995Representation}.

%
%
%

We have the Auslander's defect formula;
see \cite[Theorem]{Krause2003short}.

\begin{lemma}\label{lem:defect}
  Let
  \(
    \delta \colon 0 \to X \xto{f} Y \xto{g} Z \to 0
  \)
  be an exact sequence in $\Mod \C$, and $M \in \fp \C$.
  Then there exists a natural isomorphism
  \(
    \delta_*( D \Tr M ) \cong D \delta^*(M)
  \).
  \qed
\end{lemma}

As a consequence, the Auslander--Reiten formula follows;
compare \cite[Proposition~3.1]{AuslanderReiten1975Representation3}
and \cite[Corollaries]{Krause2003short}.

\begin{proposition}\label{prop:AR}
  Let $N$ be a $\C$-module and $M$ be a finitely presented $\C$-module. Then there exist natural isomorphisms
  \[
    \Extc ( N, D \Tr M ) \cong D \Homuc (M, N)
  \]
  and
  \[
    \Homoc ( N, D \Tr M ) \cong D \Extc (M, N).
  \]
\end{proposition}

\begin{proof}
  Let
  \(
    \delta \colon 0 \to K \to P \to N \to 0
  \)
  be an exact sequence with projective $P$.
  We observe that
  \(
    \delta_* ( D \Tr M ) = \Extc ( N, D \Tr M )
  \)
  and
  \(
    \delta^*(M) = \Homuc (M, N)
  \).
  Then Lemma~\ref{lem:defect} gives the first isomorphism.

  Let
  \(
    \delta \colon 0 \to N \to I \to K \to 0
  \)
  be an exact sequence with injective $I$. We observe that
  \(
    \delta_*( D \Tr M ) = \Homoc ( N, D \Tr M )
  \)
  and
  \(
    \delta^* (M) = \Extc (M, N)
  \).
  Then Lemma~\ref{lem:defect} gives the second isomorphism.
\end{proof}

The following result is useful in characterizing whether a morphism is projectively trivial or injectively trivial.

\begin{proposition}\label{prop:trivial}
  Let $f \colon M \to N$ be a morphism of $\C$-modules.
  \begin{enumerate}
    \item
      Assume $M$ is finitely presented. Then $f$ is projectively trivial in $\Mod \C$ if and only if
      \(
        \Extc (f, D \Tr M) = 0
      \).
    \item
      Assume $N$ is finitely copresented. Then $f$ is injectively trivial in $\Mod \C$ if and only if
      \(
        \Extc (\Tr D N, f) = 0
      \).
  \end{enumerate}
\end{proposition}

\begin{proof}
  We only prove (1).
  It is sufficient to show the sufficiency.
  Proposition~\ref{prop:AR} implies the commutative diagram
  \[\begin{tikzcd}[column sep = 7em]
    \Extc (N, D \Tr M) & \Extc (M, D \Tr M) \\
    D \Homuc (M, N)    & D \Homuc (M, M).
    \ar[from = 1-1, to = 2-1, "\cong"]
    \ar[from = 1-2, to = 2-2, "\cong"]
    \ar[from = 1-1, to = 1-2, "{\Extc(f, D\Tr M)}"]
    \ar[from = 2-1, to = 2-2, "{D \Homuc (M, \underline{f})}"]
  \end{tikzcd}\]
  Then
  \(
    \Extc (f, D \Tr M) = 0
  \)
  implies
  \(
    D \Homuc (M, \underline{f}) = 0
  \).
  Moreover,
  \(
    \Homuc (M, \underline{f}) = 0
  \)
  since $D$ is faithful.
  In particular,
  \(
    \underline{f} = \Homuc (M, \underline{f}) (\underline{\1_M}) = 0
  \)
  in $\Homuc (M, N)$.
  In other words, $f$ is projectively trivial in $\Mod \C$.
\end{proof}

\subsection{Almost split sequences}

Recall that a morphism $f \colon M \to N$ is called \emph{right almost split} if it is a non-retraction and each non-retraction $g \colon M' \to N$ factors through $f$.
Dually, $f$ is called \emph{left almost split} if it is a non-section and each non-section $g' \colon M \to N'$ factors through $f$.
An exact sequence
\(
  0 \to X \xto{g} Y \xto{f} Z \to 0
\)
is called \emph{almost split} if $f$ is right almost split and $g$ is left almost split.


We deduce the existence of almost split sequences;
compare \cite[Theorem~V.1.15]{AuslanderReitenSmalo1995Representation}.

\begin{proposition}\label{prop:ass}
  Let $M$ be an indecomposable $\C$-module.
  \begin{enumerate}
    \item
      If $M$ is finitely presented non-projective, then there exists an almost split sequence
      \[
        0 \To D \Tr M \To E \To M \To 0.
      \]
    \item
      If $M$ is finitely copresented non-injective, then there exists an almost split sequence
      \[
        0 \To M \To E \To \Tr D M \To 0.
      \]
  \end{enumerate}
\end{proposition}

\begin{proof}
  We only prove (1).
  One can choose some nonzero $\theta \in D \Enduc(M)$ vanishing on $\rad \Enduc(M)$.
  Proposition~\ref{prop:AR} implies that
  \(
    \Extc ( M, D \Tr M ) \cong D \Enduc(M)
  \).
  Assume the pre-image of $\theta$ under the isomorphism is the non-split exact sequence
  \[
      \delta \colon 0 \To D \Tr M \To E \overset{f}{\To} M \To 0.
  \]

  Claim: $f$ is right almost split.
  Indeed, it is a non-retraction since $\delta$ is non-split.
  Assume $h \colon X \to M$ is a non-retraction.
  Consider the induced map
  \[
      D \Homuc (M, \underline{h}) \colon
      D \Enduc(M) \To D \Homuc(M,X).
  \]
  Observe that $\Endc(M)$ is local. Then
  \(
    h \circ h' \in \rad \Endc(M)
  \)
  for any $h' \colon M \to X$.
  Since $\theta$ vanishing on $\rad \Enduc(M)$,
  it follows that
  \[
    D \Homuc (M, \underline{h}) (\theta) (\underline{h'})
    = (\theta \circ \Homuc(M, \underline{h})) (\underline{h'})
    = \theta (\underline{h} \circ \underline{h'}) = 0.
  \]
  Hence, 
  \(
    D \Homuc (M, \underline{h}) (\theta) = 0
  \).
  Consider the commutative diagram
  \[\begin{tikzcd}[column sep = 7em]
    \Extc ( M, D \Tr M ) & \Extc ( X, D \Tr M ) \\
    D \Enduc (M)         & D \Homuc (M,X).
    \ar[from = 1-1, to = 2-1, "\cong"]
    \ar[from = 1-2, to = 2-2, "\cong"]
    \ar[from = 1-1, to = 1-2, "{\Extc ( h, D \Tr M )}"]
    \ar[from = 2-1, to = 2-2, "{D \Homuc (M, \underline{h})}"]
  \end{tikzcd}\]
  We have that
  \(
    \Extc ( h, D \Tr M ) (\delta) = 0
  \).
  That is to say, the pullback of $\delta$ along $h$ splits.
  In other words, $h$ factors through $f$. It follows that $f$ is right almost split.

  Observe that $D \Tr M$ is indecomposable, and hence $\Endc(D \Tr M)$ is local.
  It follows that $\delta$ is an almost split sequence;
  see \cite[Proposition~I.4.4]{Auslander1978Functors}.
%
\end{proof}

\section{Generalized Auslander--Reiten duality on \texorpdfstring{$\fp \C$}{fpC}}

Let $k$ be a field.
We call a $k$-category $\C$ of type $A_\infty$ if $\Ob \C = \N$ with $\C(j,i) = 0$ for any $i<j$; compare \cite[Definition~2.2]{GanLi2015Noetherian}.
Recall that $\Mod \C$ is \emph{locally Noetherian} if $\C$-submodules of finitely generated $\C$-modules are also finitely generated; see \cite[Section~5.8]{Popescu1973Abelian}.

In this section, we assume $\C$ is a Hom-finite $k$-category of type $A_\infty$ such that $\Mod \C$ is locally Noetherian.
In particular, $\C$ is small and skeletal.

\subsection{Finitely presented modules}


We begin with the following well-known fact.

\begin{lemma}\label{lem:fg=fp}
  If $\Mod \C$ is locally Noetherian, then $\fg \C$ coincides with $\fp \C$, and is an abelian subcategory of $\Mod \C$ closed under extensions.
\end{lemma}

\begin{proof}
  Since $\Mod \C$ is locally Noetherian, every finitely generated $\C$-module is finitely presented. Then $\fg \C$ and $\fp \C$ coincide. We observe that $\fg \C$ is closed under submodules and factor modules. It follows that $\fg \C$ is an abelian subcategory of $\Mod \C$ closed under extensions.
\end{proof}


Recall that a $\C$-module $M$ is called \emph{finite dimensional} if there exist only finitely many $i \in \Ob \C$ with $M(i) \neq 0$ and these $M(i)$ are both finite dimensional. We denote by $\fd \C$ the category of finite dimensional $\C$-modules.

We mention the following observation.

\begin{lemma}\label{lem:inj-fd}
  Finitely cogenerated injective $\C$-modules are finite dimensional.
\end{lemma}

\begin{proof}
  It is sufficient to show that $D \C(-,i)$ is finite dimensional for any $i \in \Ob \C$.
  We observe that the set of $j \in \Ob \C$ with $D \C(j,i) \neq 0$ is a subset of $\set{j \in \Ob \C | j \leq i}$ which is finite, since $\C$ is of $A_\infty$ type. Since $\C$ is Hom-finite, each $D \C(j,i)$ is finite dimensional. Then the result follows.
\end{proof}

As a consequence, we obtain the following fact.

\begin{proposition}\label{prop:fcg=fcp}
  The categories $\fd \C$, $\fcg \C$ and $\fcp \C$ coincide, and are contained in $\fp \C$.
\end{proposition}

\begin{proof}
  We observe by Lemma~\ref{lem:proj-inj} that finite dimensional $\C$-modules are finitely generated and finitely cogenerated.
  Then $\fd \C$ is contained in $\fg \C$ and $\fcg \C$.

  Assume $M$ is a finitely cogenerated $\C$-module and $f \colon M \to I$ is a monomorphism with $I \in \inj \C$. Lemma~\ref{lem:inj-fd} implies that $I$ is finite dimensional. Then so is $M$. Hence $\fd \C$ and $\fcg \C$ coincide.

  Moreover, $\Cok f$ is also finite dimensional. Hence it is finitely cogenerated since $\fd \C$ and $\fcg \C$ coincide. It follows that $M$ is finitely copresented.
  Therefore $\fcg \C$ and $\fcp \C$ coincide, since $\fcp \C$ is contained in $\fcg \C$.
  Then the result follows, since $\fg \C$ and $\fp \C$ coincide by Lemma~\ref{lem:fg=fp}.
\end{proof}

Observe that $\fp \C$ is a Hom-finite Krull--Schmidt abelian category;
see Lemmas~\ref{lem:Krull} and \ref{lem:fg=fp}.
We consider its stable categories $\underline{\fp \C}$ and $\overline{\fp \C}$.
The first step is to study the projectively trivial morphisms and injectively trivial morphisms in $\fp \C$.

\begin{lemma}\label{lem:fp-trivial}
  Let $f \colon M \to N$ be a morphism in $\fp \C$.
  \begin{enumerate}
    \item
      $f$ is projectively trivial in $\fp \C$ if and only if it is projectively trivial in $\Mod \C$.
    \item
      If $M \in \fd \C$ or $N \in \fd \C$, then $f$ is injectively trivial in $\fp \C$ if and only if it is injectively trivial in $\Mod \C$.
  \end{enumerate}
\end{lemma}

\begin{proof}
  (1) We observe that $D \Tr M$ is finitely copresented and then lies in $\fp \C$ by Proposition~\ref{prop:fcg=fcp}. Then it follows from Proposition~\ref{prop:trivial}(1) that $f$ is projectively trivial in $\fp \C$ if and only if it is projectively trivial in $\Mod \C$.

  (2) If $N \in \fd \C$, it is finitely copresented by Proposition~\ref{prop:fcg=fcp}.
  Then it follows from Proposition~\ref{prop:trivial}(2) that $f$ is injectively trivial in $\fp \C$ if and only if it is injectively trivial in $\Mod \C$.

  If $M \in \fd \C$, its injective envelope in $\Mod \C$ lies in $\fp \C$ by Proposition~\ref{prop:fcg=fcp}. Then the result follows from Lemma~\ref{lem:trivial}.
\end{proof}

As a consequence of Lemma~\ref{lem:fp-trivial}, we have that $\underline{\fp \C} = \fpu \C$ and $\fdo \C$ is a full subcategory of $\overline{\fp \C}$.
Here, $\underline{\fp \C}$ is the projectively stable category of $\fp \C$, and $\fpu \C$ is the full subcategory of $\Modu \C$ formed by finitely presented $\C$-modules.

Assume $f \colon M \to N$ is an injectively trivial morphism in $\fp \C$ such that $N \in \fd \C$ but $M \notin \fd \C$.
We mention that $f$ needs not factor through some injective object in $\fp \C$.
But Lemma~\ref{lem:fp-trivial}(2) implies that $f$ is injectively trivial in $\Mod \C$. It follows from Lemma~\ref{lem:trivial} that $f$ factors through some injective $\C$-module $I$. Here, $I$ needs not lie in $\fp \C$.

There may exist some injectively trivial morphisms $f \colon M \to N$ in $\fp \C$, such that $M, N \notin \fd \C$.
In this case, we have no idea about properties of these $f$, including whether $f$ factors through some injective object in $\fp \C$ or $\Mod \C$.

\subsection{Generalized Auslander--Reiten duality}

Recall from \cite[Section~2]{Jiao2018generalized} that the \emph{generalized Auslander--Reiten duality} on $\fp \C$ consists of a pair of full categories $(\fp \C)_r$ and $(\fp \C)_l$, and a pair of functors
\[
  \tau \colon \underline{(\fp \C)_r} \To \overline{(\fp \C)_l}
  \quad \mbox{and} \quad
  \tau \colon \overline{(\fp \C)_l} \To \underline{(\fp \C)_r}.
\]
Here, $\underline{(\fp \C)_r}$ is the image of $(\fp \C)_r$ under the factor functor $\fp \C \to \fpu \C$, and $\overline{(\fp \C)_l}$ is the image of $(\fp \C)_l$ under the factor functor $\fp \C \to \overline{\fp \C}$.

The subcategories $(\fp \C)_r$ and $(\fp \C)_l$ are given as follows
\[
  (\fp \C)_r =
  \set{M \in \fp \C \middle|
    D \Extc(M,-) \colon \overline{\fp \C} \to \Mod k
    \mbox{ is representable}}
\]
and
\[
  (\fp \C)_l =
  \set{M \in \fp \C \middle|
    D \Extc(-,M) \colon \underline{\fp \C} \to \Mod k
    \mbox{ is representable}}.
\]
We mention that $(\fp \C)_r$ and $(\fp \C)_l$ are both additive.


For any $M \in (\fp \C)_l$ and $N \in \fp \C$, there exists a natural isomorphism
\[
  \Homuc (\tau^- M, N) \cong D \Extc (N, M).
\]
For any $N \in (\fp \C)_r$ and $M \in \fp \C$, there exists a natural isomorphism
\[
  \Homoc (M, \tau N) \cong D \Extc (N, M).
\]
Moreover, the functors $\tau$ and $\tau^-$ are mutually quasi-inverse equivalences.
They are called the \emph{generalized Auslander--Reiten translation functors}.

We mention the following characterizations for objects in $(\fp \C)_r$ and $(\fp \C)_l$;
see \cite[Proposition~2.4]{Jiao2018generalized}.

\begin{lemma}\label{lem:Cr-Cl}
  Let $M$ be an indecomposable object in $\fp \C$.
  \begin{enumerate}
    \item
      If $M$ is non-projective in $\fp\C$, then $M$ lies in $(\fp \C)_r$ if and only if there exists an almost split sequence ending at $M$.
    \item
      If $M$ is non-injective in $\fp\C$, then $M$ lies in $(\fp \C)_l$ if and only if there exists an almost split sequence starting at $M$.
    \qed
  \end{enumerate}
\end{lemma}

Considering the above lemma, it is necessary to study the almost split sequences in $\fp \C$.

\begin{lemma}\label{lem:ass-fp}
  An exact sequence in $\fp \C$ is almost split if and only if it is an almost split sequence in $\Mod \C$.
\end{lemma}

\begin{proof}
  The sufficiency is immediate. For the necessary, we assume
  \[
    \delta \colon 0 \To M \To E \To N \To 0
  \]
  is an almost split sequence in $\fp \C$. We observe that $N$ is a finitely presented non-projective $\C$-module.
  Then there exists an almost split sequence
  \[
    \epsilon \colon 0 \To D \Tr N \To E' \To N \To 0
  \]
  in $\Mod \C$ by Proposition~\ref{prop:ass}(1). We observe that $D \Tr N$ is finitely copresented.
  Proposition~\ref{prop:fcg=fcp} implies that $D \Tr N$ is finitely presented, and hence $\epsilon$ lies in $\fp \C$. Then $\epsilon$ is an almost split sequence in $\fp \C$, and hence is isomorphic to $\delta$. It follows that $\delta$ is an almost split sequence in $\Mod \C$.
\end{proof}

The following result gives the generalized Auslander--Reiten duality on $\fp \C$.
It is analogous to \cite[Proposition~4.4]{Jiao2018generalized}.

\begin{theorem}\label{thm:gAR}
  Let $\C$ be a Hom-finite category of type $A_\infty$ such that $\Mod \C$ is locally Noetherian. Then
  \[
    ( \fp \C )_r = \fp \C
  \]
  and
  \[
    ( \fp \C )_l = \add
    \left(
    \fd \C \cup \set{\mbox{injective objects in $\fp \C$}}
    \right).
  \]
  Moreover, the functors $D \Tr$ and $\Tr D$ induce the generalized Auslander--Reiten translation functors.
\end{theorem}

\begin{proof}
  We observe that projective objects lie in $(\fp \C)_r$.
  Let $M$ be an indecomposable non-projective object in $\fp \C$.
  Proposition~\ref{prop:ass}(1) gives an almost split sequence
  \[
    \delta \colon 0 \To D \Tr M \To E \To M \To 0.
  \]
  We observe by Proposition~\ref{prop:fcg=fcp} that $D \Tr M$ is finitely presented. Then $\delta$ is an almost split sequence in $\fp \C$. Lemma~\ref{lem:Cr-Cl}(1) implies that $M$ lies in $(\fp \C)_r$.
  Then the first equality follows.

  Observe that injective objects lie in $(\fp \C)_l$.
  Let $N$ be a finite dimensional indecomposable non-injective object in $\fp \C$.
  We observe by Proposition~\ref{prop:fcg=fcp} that $N$ is finitely copresented. Proposition~\ref{prop:ass}(2) gives an almost split sequence starting at $N$, which lies in $\fp \C$. Lemma~\ref{lem:Cr-Cl}(2) implies that $N$ lies in $(\fp \C)_l$.

  On the other hand, let $N$ be an indecomposable non-injective object lying in $(\fp \C)_l$.
  Lemma~\ref{lem:Cr-Cl}(2) implies that there exists an almost split sequence
  \[
    \delta \colon 0 \To N \To E \To M \To 0
  \]
  in $\fp \C$.
  Lemma~\ref{lem:ass-fp} implies that $\delta$ is an almost split sequence in $\Mod \C$. Since $M$ is non-projective, we observe by Proposition~\ref{prop:ass}(1) that $N \cong D \Tr M$ and is finitely copresented. Proposition~\ref{prop:fcg=fcp} implies that $N$ is finite dimensional.
  Then the second equality follows.


  We observe that $\fdo \C$ is a dense full subcategory of $\overline{(\fp \C)_l}$, since any injective object becomes zero in $\overline{\fp \C}$.
  Then $D \Tr$ and $\Tr D$ induce functors
  \[
    \tau \colon \fpu \C \To \overline{(\fp \C)_l}
    \quad \mbox{and} \quad
    \tau^- \colon \overline{(\fp \C)_l} \To \fpu \C,
  \]
  which are mutually quasi-inverse equivalences.

  Proposition~\ref{prop:AR} gives natural isomorphisms
  \[
    \phi \colon \Homoc(M, \tau N) \overset\cong\To D \Extc(N, M)
  \]
  for any $M, N \in \fp \C$, and
  \[
    \psi \colon \Homuc(\tau^- M, N) \overset\cong\To D \Extc(N, M)
  \]
  for any $M \in (\fd \C)_l$ and $N \in \fp \C$.
  Here, we mention that $\Homoc(M, \tau N)$ is the Hom-set in $\Modo \C$ by Proposition~\ref{prop:trivial}(2).
  Then the result follows.
%
\end{proof}


\section{Applications}

Let $k$ be a field. We will apply the previous results to $\FI$, $\VI$ and some certain infinite quivers in this section.

\subsection{Quivers}

Let $Q = (Q_0, Q_1)$ be a quiver, where $Q_0$ is the set of vertices and $Q_1$ is the set of arrows. For any arrow $\alpha \colon a \to b$, we denote by $s(\alpha) = a$ its source and by $t(\alpha) = b$ its target.

Every vertex $a$ is associated with a trivial path (of length 0) $e_a$ with $s(e_a) = a = t(e_a)$.
A path $p$ of length $l \geq 1$ is a sequence of arrows $\alpha_l \cdots \alpha_2 \alpha_1$ that $s(\alpha_{i+1}) = t(\alpha_i)$ for any $1 \leq i < l$. We set $s(p) = s(\alpha_1)$ and $t(p) = t(\alpha_l)$.
For any path $p$, we have $e_{t(p)} p = p = p e_{s(p)}$.
For any vertices $a$ and $b$ we denote by $Q(a, b)$ the set of paths $p$ with $s(p) = a$ and $t(p) = b$.

In this subsection, we assume $Q_0 = \N$ and $0< \norm{Q(i,j)} < \infty$ and $\norm{Q(j,i)} = 0$ for any $0 \leq i < j$.
In particular, $Q$ has a subquiver of the form
\[
  \begin{tikzcd}
    \underset0\circ \rar & \underset1\circ \rar & \underset2\circ \rar
    & \underset3\circ \rar & \cdots.
  \end{tikzcd}
\]

View $Q$ as a small category, and let $\C$ be its $k$-linearization;
see \cite[Section2.1]{GabrielRoiter1992Representations}.
Then $\C$ is a Hom-finite $k$-category of type $A_\infty$.
The category of representations of $Q$ is isomorphic to $\Mod \C$.
Denote $P_a = \C(a, -)$ and $I_a = D \C(-, a)$ for any $a \in \Ob \C$.
It is well known that $\Mod \C$ is hereditary;
see \cite[Section~8.2]{GabrielRoiter1992Representations}.

We mention the following fact.

\begin{lemma}\label{lem:fp-abel}
  The category $\fp \C$ is a hereditary abelian subcategory of $\Mod \C$ closed under extensions.
\end{lemma}

\begin{proof}
  Let $f \colon P \to P'$ be a morphism in $\proj \C$.
  Since $\Mod \C$ is hereditary, then $\Im f$ is projective. Therefore, the induced exact sequence
  \[
    0 \To \Ker f \To P \To \Im f \To 0
  \]
  splits, and hence $\Ker f \in \proj \C$.
  Then the result follows from \cite[Proposition~2.1]{Auslander1966Coherent} and the horseshoe lemma.
\end{proof}

We mention that $\Mod \C$ needs not be locally Noetherian in general, even though $\fp \C$ is abelian by Lemma~\ref{lem:fp-abel}.
See the following example.

\begin{example}
  Assume $Q$ is the following quiver.
  
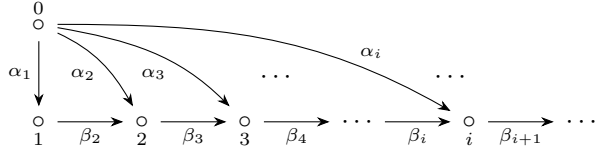
\[\begin{tikzcd}[sep = 2.5em]
    \overset0\circ \\
    \underset1\circ & \underset2\circ & \underset3\circ & \cdots & \underset{i}\circ & \cdots
    \ar[from = 1-1, to = 2-1, "\alpha_1"']
    \ar[from = 1-1, to = 2-2, out = -25, in = 120, "\alpha_2"']
    \ar[from = 1-1, to = 2-3, out = -10, in = 135, "\alpha_3"' pos = 0.63]
    \ar[from = 1-1, to = 2-4, out =   0, in = 150, phantom, "\cdots" pos = 0.8]
    \ar[from = 1-1, to = 2-5, out =   0, in = 150, "\alpha_i" near end]
    \ar[from = 1-1, to = 2-6, out =   0, in = 160, phantom, "\cdots" pos = 0.8]
    \ar[from = 2-1, to = 2-2, "\beta_2"']
    \ar[from = 2-2, to = 2-3, "\beta_3"']
    \ar[from = 2-3, to = 2-4, "\beta_4"']
    \ar[from = 2-4, to = 2-5, "\beta_i"']
    \ar[from = 2-5, to = 2-6, "\beta_{i+1}"']
  \end{tikzcd}\]
  We have the injection
  \[
    (f_1, f_2, \dots, f_i, \dots) \colon \bigoplus_{i \geq 1} P_i \To P_0.
  \]
  Here, $f_i$ is induced by $\alpha_i$. It follows that $\Mod \C$ is not locally Noetherian.
\end{example}

Recall that $Q$ is called \emph{uniformly interval finite} if there exists some integer $N$ such that $\norm{Q(a, b)} \leq N$ for any $a, b \in Q_0$;
see \cite[Definition~2.3]{Jiao2019Projective}.
We have the following characterization.

\begin{proposition}\label{prop:Q-Noeth}
  The category $\Mod \C$ is locally Noetherian if and only if $Q$ is uniformly interval finite.
\end{proposition}

\begin{proof}
  We observe that
  \(
    \norm{Q(i, j)} \leq \norm{Q(i', j')}
  \)
  for any $i' \leq i$ and $j' \geq j$, since $Q(i',i)$ and $Q(j,j')$ are nonempty.
  Then $Q$ is uniformly interval finite if and only if
  \(
    \set{\norm{Q(0, j)} | j \in Q_0}
  \)
  is bounded.

  If
  \(
    \set{\norm{Q(0, j)} | j \in Q_0}
  \)
  is bounded, there exists some $n \in \N$ such that
  \(
    \norm{Q(0, n)} = \norm{Q(0, j)}
  \)
  for any $j \geq n$.
  Then for any $i \geq 0$, we have that $\dim P_i (j)$ coincide for all $j \geq \max \set{i, n}$.

  For any submodule $M$ of $P_i$, there exists some $m \in \N$ such that $\dim M(i) = \dim M(m)$ for any $i \geq m$.
  Consider the submodule $M'$ of $M$ such that
  \[
    M'(i) =
    \left\{
    \begin{array}{@{}ll@{}}
      M(i), & \mbox{if } i \geq m, \\
      0,    & \mbox{if } i < m.
    \end{array}
    \right.
  \]
  We observe that
  \(
    M' \cong P_m^{\oplus \dim M(m)}
  \)
  and $M / M'$ is finite dimensional.
  It follows that $M$ is finitely generated.

  We observe that Noetherian property is closed under finite direct sums and factor modules. Then $\Mod \C$ is locally Noetherian.

  If
  \(
    \set{\norm{Q(0, j)} | j \in Q_0}
  \)
  is unbounded, we consider $P_0$.
  There exists some $i_1>1$ such that
  \(
    \norm{Q(0, i_1)} > \norm{Q(0, i_1-1)} \geq 1
  \).
  Moreover, there exists some $i_2$ such that
  \(
    \norm{Q(0, i_2)} > \norm{Q(0, i_1)}
  \).
  Then at least two paths in $Q(0, i_2)$ are not the form $u p_1$ for any $u \in Q(i_1, i_2)$. We denote one of them by $p_2$.

  Inductively, for any $j \geq 2$, there exists some $i_j$ such that
  \(
    \norm{Q(0, i_j)} > \norm{Q(0, i_{j-1})}
  \).
  Then at least two paths in $Q(0, i_j)$ are not the form $u p_r$ for any $1 \leq r < j$ and $u \in Q(i_r, i_j)$. Denote one of them by $p_j$.

  We then obtain the monomorphism
  \[
    (f_1, f_2, \dots, f_j, \dots) \colon \bigoplus_{j \geq 1} P_{i_j} \To P_0,
  \]
  where $f_j$ is induced by $p_j$.
  It follows that $\Mod \C$ is not locally Noetherian.
\end{proof}

We study the generalized Auslander--Reiten duality on $\fp \C$ when $Q$ is uniformly interval finite.

For each $a \in Q_0$, we denote by $Q(a,\infty)$ the set of infinite sequences of arrows
\(
  \cdots\alpha_i\cdots\alpha_2\alpha_1
\),
such that $s(\alpha_1) = a$ and $s(\alpha_{i+1}) = t(\alpha_i)$ for any $i \geq 1$.

We introduce the representation $Y$ as follows.
For each vertex $a$, let
\(
  Y(a) = \Hom_k (\bigoplus_{p \in Q(a,\infty)} k p, k)
\).
For each arrow $\alpha \colon a \to b$, let
\(
  Y(\alpha) \colon Y(a) \to Y(b)
\)
be given by
\(
  Y(\alpha)(f)(q) = f(q\alpha)
\),
for any $f \in Y(a)$ and $q \in Q(b,\infty)$.

We mention that $Y$ is an indecomposable injective object in $\fp \C$.
Moreover, we have the following characterization of indecomposable injective objects in $\fp \C$;
see \cite[Theorem~3.11]{Jiao2019Injective}.

\begin{lemma}\label{lem:Q-inj}
  If $Q$ is uniformly interval finite, then
  \[
    \set{Y} \cup \set{I_a | a \in Q_0}
  \]
  is a complete set of indecomposable injective objects in $\fp \C$.
  \qed
\end{lemma}

Then we can make the subcategory $(\fp \C)_l$ more explicit.

\begin{proposition}\label{prop:Q-Cl}
  Assume $Q$ is uniformly interval finite. Then
  \[
    (\fp \C)_l = \add
    \left(
    \fd \C \cup \set{Y}
    \right).
  \]
\end{proposition}

\begin{proof}
  We observe by Theorem~\ref{thm:gAR} that an indecomposable object in $(\fp \C)_l$ is finite dimensional or an injective object in $\fp \C$.
  Lemma~\ref{lem:Q-inj} implies that an indecomposable injective object in $\fp \C$ is either $Y$ or $I_a$ for some $a \in Q_0$.
  Since every $I_a$ is finite dimensional, then the equality follows.
\end{proof}

\begin{example}
  Assume $Q$ is the following quiver.
  \[
    \begin{tikzcd}
      \underset0\circ \rar["\alpha_1"]
      & \underset1\circ \rar["\alpha_2"]
      & \underset2\circ \rar["\alpha_3"]
      & \underset3\circ \rar["\alpha_4"]
      & \cdots
    \end{tikzcd}
  \]
  We observe that $Q$ is uniformly interval finite.
  Then $\Mod \C$ is locally Noetherian by Proposition~\ref{prop:Q-Noeth}.

  For any $j \geq i \geq 0$, we denote the indecomposable $\C$-module
  \[\begin{tikzcd}[sep=scriptsize]
    X_{ij} \colon 0 \rar["0"]
    & \cdots \rar["0"]
    & 0 \rar["0"]
    & \underset{i}{k} \rar["1"]
    & \cdots \rar["1"]
    & \underset{j}{k} \rar["0"]
    & 0 \rar["0"]
    & \cdots.
  \end{tikzcd}\]
  We observe that
  \[
    \set{X_{ij} \middle| j \geq i \geq 0} \cup \set{P_i \middle| i \geq 0}
  \]
  is a complete set of indecomposable $\C$-modules.
  Here, $P_0 \cong Y$ and $X_{0j} \cong I_j$ for any $j \geq 0$.
  It follows from Theorem~\ref{thm:gAR} and Proposition~\ref{prop:Q-Cl} that
  \[
    (\fp \C)_r = \fp \C = \add
    \left(
      \set{X_{ij} \middle| j \geq i \geq 0} \cup \set{P_i \middle| i \geq 0}
    \right)
  \]
  and
  \[
    ( \fp \C )_l = \add
    \left(
      \set{X_{ij} \middle| j \geq i \geq 0} \cup \set{P_0}
    \right).
  \]
\end{example}

\subsection{FI and VI}

Assume the field $k$ is of characteristic 0.
Recall that $\FI$ is the category whose objects are finite sets and morphisms are injections, and $\VI$ is the one whose objects are finite dimensional vector spaces over a finite field $\F_q$ and morphisms are $\F_q$-linear injections.

Let $G$ be a finite group.
Recall from \cite[Definition~1.1]{GanLi2015Coinduction} that $\FI_G$ is the category whose objects are finite sets, and $\FI_G (S, T)$ is the set of pairs $(f, g)$ where $f \colon S \to T$ is an injection and $g \colon S \to G$ is an arbitrary map.
The composition of $(f, g) \in \FI_G (S, T)$ and $(f', g') \in \FI_G (T, T')$ is given by
\[
  (f', g') \circ (f, g) = (f' \circ f, g''),
\]
where
\(
  g''(x) = g'(f(x)) \cdot g(x)
\)
for any $x \in S$.
We observe that $\FI_G$ is isomorphic to $\FI$ if $G$ is the trivial group.


Given a skeleton of $\FI_G$ (or $\VI$), we will denote every object by its cardinal (or its $\F_q$-dimension) $n \in \N$.
Let $\C$ be the $k$-linearization of the skeleton.
Then $\C$ is a Hom-finite $k$-category of type $A_\infty$.
The category of $\FI_G$-modules (or $\VI$-modules) over $k$ is isomorphic to $\Mod \C$.

The following result follows from \cite[Theorem~3.7]{GanLi2015Noetherian}.

\begin{lemma}\label{lem:FI-Noeth}
  The category $\Mod \C$ is locally Noetherian.
  \qed
\end{lemma}

We will study the generalized Auslander--Reiten duality on $\fp \C$.

The following characterization of injective objects in $\fp \C$ is counter-intuitive;
see \cite[Theorems~1.5 and 1.7]{GanLi2015Coinduction}
and \cite[Theorems~1.9 and 5.23]{Nagpal2019VI}.

\begin{lemma}\label{lem:FI-inj}
  Every finitely generated projective $\C$-module is an injective object in $\fp \C$,
  and every indecomposable injective object in $\fp \C$ lies in either $\inj \C$ or $\proj \C$.
  \qed
\end{lemma}

The above fact implies that any projectively trivial morphism in $\fp \C$ is also an injectively trivial morphism in $\fp \C$.
Therefore, $\overline{\fp \C}$ is a factor category of $\fpu \C$.
But, Theorem~\ref{thm:gAR} implies that $\fpu \C$ is equivalent to the full subcategory $\fdo \C$ of $\overline{\fp \C}$. It is somehow surprising.

We can make the subcategory $(\fp \C)_l$ more explicit.

\begin{proposition}\label{prop:FI-Cl}
  Let $\C$ be the $k$-linearization of a skeleton of $\FI_G$ or $\VI$. Then
  \[
    (\fp \C)_l = \add
    \left(
    \fd \C \cup \proj \C
    \right).
  \]
\end{proposition}

\begin{proof}
  We observe by Theorem~\ref{thm:gAR} that an indecomposable object in $(\fp \C)_l$ is finite dimensional or an injective object in $\fp \C$.
  Lemma~\ref{lem:FI-inj} implies that an indecomposable injective object in $\fp \C$ lies in either $\inj \C$ or $\proj \C$.
  Since $\inj \C$ is contained in $\fd \C$, then the equality follows.
\end{proof}

\section*{Acknowledgements}

The author is grateful to Professor Xiao-Wu Chen for many helpful suggestions,
and thanks the referee for pointing out some errors and helpful comments.

This work was supported by the National Natural Science Foundation of China (Grant No.~11901545).


\begin{thebibliography}{LNP13}

\bibitem[A66]{Auslander1966Coherent}
  M.~Auslander,
  \emph{Coherent functors},
  Proc. {C}onf. {C}ategorical {A}lgebra ({L}a {J}olla, {C}alif., 1965),
  Springer, New York, 1966, pp.~189--231.

\bibitem[A78]{Auslander1978Functors}
  M.~Auslander,
  \emph{Functors and morphisms determined by objects},
  Representation theory of algebras ({P}roc. {C}onf., {T}emple {U}niv., {P}hiladelphia, {P}a., 1976),
  Lecture Notes in Pure and Applied Mathematics, vol.~37,
  Dekker, New York, 1978, pp.~1--244.

\bibitem[AR75]{AuslanderReiten1975Representation3}
  M.~Auslander and I.~Reiten,
  \emph{Representation theory of {A}rtin algebras {III}: {A}lmost split sequences},
  Comm. Algebra \textbf{3} (1975), 239--294.

\bibitem[ARS95]{AuslanderReitenSmalo1995Representation}
  M.~Auslander, I.~Reiten, and S.~O. Smal\o,
  \emph{Representation theory of {A}rtin algebras},
  Cambridge Studies in Advanced Mathematics, vol.~36,
  Cambridge University Press, Cambridge, 1995.

\bibitem[CEF15]{ChurchEllenbergFarb2015FI}
  T.~Church, J.~S. Ellenberg, and B.~Farb,
  \emph{{FI}-modules and stability for representations of symmetric groups},
  Duke Math. J. \textbf{164} (2015), 1833--1910.

\bibitem[GR92]{GabrielRoiter1992Representations}
  P.~Gabriel and A.~V. Roiter,
  \emph{Representations of finite-dimensional algebras},
  Algebra {VIII}, Encyclopaedia of Mathematical Sciences, vol.~73,
  Springer, Berlin, 1992, pp.~1--177.

\bibitem[GL15a]{GanLi2015Coinduction}
  W.~L. Gan and L.~Li,
  \emph{Coinduction functor in representation stability theory},
  J. Lond. Math. Soc. (2) \textbf{92} (2015), 689--711.

\bibitem[GL15b]{GanLi2015Noetherian}
  W.~L. Gan and L.~Li,
  \emph{Noetherian property of infinite {EI} categories},
  New York J. Math. \textbf{21} (2015), 369--382.

\bibitem[J18]{Jiao2018generalized}
  P.~Jiao,
  \emph{The generalized {A}uslander-{R}eiten duality on an exact category},
  J. Algebra Appl. \textbf{17} (2018), 1850227.

\bibitem[J19a]{Jiao2019Projective}
  P.~Jiao,
  \emph{Projective objects in the category of pointwise finite dimensional representations of an interval finite quiver},
  Forum Math. \textbf{31} (2019), 1331--1349.

\bibitem[J19b]{Jiao2019Injective}
  P.~Jiao,
  \emph{Injective objects in the category of finitely presented representations of an interval finite quiver},
  Ark. Mat. \textbf{57} (2019), 381--396.

\bibitem[K03]{Krause2003short}
  H.~Krause,
  \emph{A short proof for {A}uslander's defect formula},
  Linear Algebra Appl. \textbf{365} (2003), 267--270.

\bibitem[K15]{Krause2015Krull}
  H.~Krause,
  \emph{Krull-{S}chmidt categories and projective covers},
  Expo. Math. \textbf{33} (2015), 535--549.

\bibitem[LZ04]{LenzingZuazua2004Auslander}
  H.~Lenzing and R.~Zuazua,
  \emph{Auslander-{R}eiten duality for abelian categories},
  Bol. Soc. Mat. Mexicana (3) \textbf{10} (2004), 169--177.

\bibitem[LNP13]{LiuNgPaquette2013Almost}
  S.~Liu, P.~Ng, and C.~Paquette,
  \emph{Almost split sequences and approximations},
  Algebr. Represent. Theory \textbf{16} (2013), 1809--1827.

\bibitem[N19]{Nagpal2019VI}
  R.~Nagpal,
  \emph{{VI}-modules in nondescribing characteristic, part {I}},
  Algebra Number Theory \textbf{13} (2019), 2151--2189.

\bibitem[P73]{Popescu1973Abelian}
  N.~Popescu,
  \emph{Abelian categories with applications to rings and modules},
  London Mathematical Society Monographs, vol.~3,
  Academic Press, London-New York, 1973.

\end{thebibliography}

\end{document}